\documentclass[11pt,english]{smfart} 
\usepackage[T1]{fontenc}
\usepackage[english,french]{babel}
\usepackage{amscd}
\usepackage{amssymb,url,xspace,smfthm}
\usepackage{amsbsy,bm} 
\usepackage{paralist}
\usepackage{datetime} 
\usepackage{colortbl}
\usepackage{color}
\usepackage[dvipsnames]{xcolor}
\usepackage{mathrsfs}

\usepackage[T1]{fontenc}
\usepackage{newtxtext}
\usepackage{newtxmath}
\usepackage{textcomp}
\usepackage[colorlinks,	colorlinks,
linkcolor=BrickRed,
citecolor=Green,
urlcolor=Cerulean,hypertexnames=true]{hyperref} 
\usepackage{mathtools}
\usepackage{euscript}
\usepackage[all]{xy}
\usepackage{hyperref}
\usepackage{graphicx} 
\usepackage[text={150mm,251mm},centering, marginparwidth=75pt]{geometry} 
\usepackage{comment}  
\usepackage{cite}  
\usepackage{thmtools}
\usepackage{enumitem}
\usepackage{letltxmacro} 
\usepackage{nameref}
\usepackage{cleveref}
\usepackage{calc}
\usepackage{interval}
\usepackage{manfnt}
\usepackage{tikz-cd}  
\usepackage[utf8]{inputenc}
\usepackage{marginnote}
\usepackage{smfhyperref}
\usepackage[hyperpageref]{backref}
\renewcommand{\backref}[1]{$\uparrow$~#1}

\usepackage{faktor} 
\usepackage{scalerel}

\makeatletter     
\newcommand{\crefnames}[3]{%
	\@for\next:=#1\do{%
		\expandafter\crefname\expandafter{\next}{#2}{#3}%
	}%
}
\makeatother

\setlist[itemize]{wide = 0pt, labelwidth = 2em, labelsep*=0em, itemindent = 0pt, leftmargin = \dimexpr\labelwidth + \labelsep\relax, noitemsep,topsep = 1ex,}
\setlist[enumerate]{wide = 0pt, labelwidth = 2em, labelsep*=0em, itemindent = 0pt, leftmargin = \dimexpr\labelwidth + \labelsep\relax, noitemsep,topsep = 1ex}

\theoremstyle{plain}
\newtheorem{thmx}{Theorem}
\renewcommand{\thethmx}{\Alph{thmx}} 
\newtheorem{theorem}{Theorem}[section]  
\newtheorem{lemma}[theorem]{Lemma}
\newtheorem{claim}[theorem]{Claim} 
\newtheorem{proposition}[theorem]{Proposition}

\theoremstyle{definition}
\newtheorem{definition}[theorem]{Definition} 
\theoremstyle{remark}
\newtheorem{remark}[theorem]{Remark}

\newcommand{\GL}{{\rm GL}}
\newcommand{\bZ}{\mathbb{Z}}
\newcommand{\bF}{\mathbb{F}}
\newcommand{\mxp}{M_{\rm B}(X,N)_{\bF_p}}
\newcommand{\mzp}{M_{\rm B}(Z,N)_{\bF_p}}
\newcommand{\myp}{M_{\rm B}(Y,N)_{\bF_p}}

\numberwithin{equation}{section}  

\theoremstyle{plain}
\newlist{thmlist}{enumerate}{1}
\setlist[thmlist]{wide = 0pt, labelwidth = 2em, labelsep*=0em, itemindent = 0pt, leftmargin = \dimexpr\labelwidth + \labelsep\relax, noitemsep,topsep = 1ex, font=\normalfont, label=(\roman*), ref=\thetheorem.(\roman{thmlisti})}

\addtotheorempostheadhook[theorem]{\crefalias{thmlisti}{thm}}

\addtotheorempostheadhook[assumpsion]{\crefalias{thmlisti}{assumption}}

\addtotheorempostheadhook[corollary]{\crefalias{thmlisti}{cor}}

\addtotheorempostheadhook[proposition]{\crefalias{thmlisti}{proposition}}

\addtotheorempostheadhook[definition]{\crefalias{thmlisti}{dfn}}

\addtotheorempostheadhook[lemma]{\crefalias{thmlisti}{lem}}
\addtotheorempostheadhook[main]{\crefalias{thmlisti}{main}}

\addtotheorempostheadhook[remark]{\crefalias{thmlisti}{rem}}

\newlist{thmenum}{enumerate}{1} 
\setlist[thmenum]{wide = 0pt, labelwidth = 2em, labelsep*=0em, itemindent = 0pt, leftmargin = \dimexpr\labelwidth + \labelsep\relax, noitemsep,topsep = 1ex, font=\normalfont, label=(\roman*), ref=\thethmx.(\roman{thmenumi})}
\crefalias{thmenumi}{thmx} 

\newlist{corlist}{enumerate}{1} 
\setlist[corlist]{wide = 0pt, labelwidth = 2em, labelsep*=0em, itemindent = 0pt, leftmargin = \dimexpr\labelwidth + \labelsep\relax, noitemsep,topsep = 1ex, font=\normalfont, label=(\roman*), ref=\thecorx.(\roman{corlisti})}
\crefalias{corlisti}{corx} 

\crefname{lemma}{Lemma}{Lemmas} 
\crefname{conjecture}{Conjecture}{Conjectures}
\crefname{theorem}{Theorem}{Theorems}
\crefname{proposition}{Proposition}{Propositions}
\crefname{definition}{Definition}{Definitions}
\crefname{remark}{Remark}{Remarks}
\crefname{corollary}{Corollary}{Corollaries}
\crefname{corx}{Corollary}{Corollaries}
\crefname{problem}{Problem}{Problems}
\crefname{thmx}{Theorem}{Theorems}
\crefname{claim}{Claim}{Claims}
\crefname{assumption}{Assumption}{Assumptions}
\crefname{main}{Main Theorem}{Main Theorems}

\newcommand{\C}{\mathbb{C}}
\newcommand{\R}{\mathbb{R}}
\newcommand{\Q}{\mathbb{Q}}
\newcommand{\A}{\mathbb{A}}
\newcommand{\K}{\mathbb{K}}
\newcommand{\spec}{{\rm Spec}}

\begin{document}

\title[Morphisms between character varieties]{Quasi-finiteness of morphisms between character varieties}

{
\author[Y. Deng]{Ya Deng}
\email{ya.deng@math.cnrs.fr}
\address{CNRS, Institut \'Elie Cartan de Lorraine, Universit\'e de Lorraine, F-54000 Nancy,
	France.}
\urladdr{https://ydeng.perso.math.cnrs.fr} 
}

{
\author[Y. Liu]{Yuan Liu}
\email{yuanliu@ustc.edu.cn}
\address{Institute of Geometry and Physics, University of Science and Technology of China, Hefei, China.} 
}

\date{\today}
\begin{abstract}
 Let $f: Y\to X$ be a morphism between   smooth complex quasi-projective varieties and $Z$ be the closure of $f(Y)$   with $\iota: Z\to X$ the inclusion map.  We prove that: 
 \begin{itemize}
  	\item for any field $\K$, there exist  finitely many semisimple representations $\{\tau_i:\pi_1(Z)\to {\rm GL}_N(\overline{k})\}_{i=1,\ldots,\ell}$ with $k\subset \K$ the minimal field contained in $\K$ such that  if $\varrho:\pi_1(X)\to \GL_{N}(\K)$ is  any  representation satisfying $[f^*\varrho]=1$,    then  $[\iota^*\varrho]=[\tau_i]$ for some $i$.  
 	\item  The induced morphism between $\GL_{N}$-character varieties (of any characteristic) of $\pi_1(X)$ and $\pi_1(Y)$   is quasi-finite if ${\rm Im}[\pi_1(Z)\to \pi_1(X)]$ is a finite index subgroup of $\pi_1(X)$. 
 \end{itemize} These results extend the main results by Lasell \cite{Lasell1995} and   Lasell-Ramachandran \cite{LR1996} from smooth complex projective varieties to quasi-projective cases with richer structures.   
\end{abstract}
\maketitle

\section{Main results} 
Let $f: Y\to X$ be a morphism between smooth complex projective varieties, and let $Z$ be the closure of $f(Y)$. By the work of Deligne \cite[Proposition 8.2.7]{DeligneIII},  the kernel of the map from $H^*(X, \C)$ to
$H^*(Y,\C)$ induced by $f$  is the same as that of   $H^*(X, \C)\to H^*(Z,\C)$.  In this paper, we will study this phenomenon in the context of non-abelian Hodge theories.  
Our first theorem is as follows.
\begin{thmx}[=\Cref{thm:trivial,thm:trivial2}]\label{main}
Let $f: Y\to X$ be a morphism between   smooth quasi-projective varieties and $Z$ be the closure of the image  $f(Y)$. Let $\iota:Z\hookrightarrow X$ be the natural inclusion.  Let $K$ be any  field and let $k\subset K$ be the minimal field contained in $K$, which is $\Q$ if ${\rm char}\, K=0$ and is $\bF_p$ if ${\rm char}\, K=p>0$. Then there exists finitely many semisimple representations $\{\tau_i:\pi_1(Z)\to {\rm GL}_N(\overline{k})\}_{i=1,\ldots,\ell}$ with  $\tau_i(\pi_1(Z))$ being finite groups such that
	 \begin{thmlist}
	 	\item    
	 	if  ${\rm char}\, K=0$ and $\varrho:\pi_1(X)\to {\rm GL}_N(K)$ is a semisimple representation with  $f^*\varrho=1$,  then    $\iota^*\varrho:\pi_1(Z)\to {\rm GL}_N(K)$ is conjugate to some $\tau_i:\pi_1(Z)\to \GL_{N}(\overline{\Q})$. In particular,  $\iota^*\varrho(\pi_1(Z))$ is a finite group.  
	 	\item 	 	if   ${\rm char}\, K=p>0$ and  $\varrho:\pi_1(X)\to {\rm GL}_N(K)$ is a linear  representation such that the semisimplification of $f^*\varrho$ is trivial, then  the semisimplification of $\iota^*\varrho:\pi_1(Z)\to {\rm GL}_N(K)$ is conjugate to some $\tau_i:\pi_1(Z)\to \GL_{N}(\overline{\bF}_p)$, and $\iota^*\varrho(\pi_1(Z))$ is a finite group.  
	 \end{thmlist}  
\end{thmx}
We remark that the above theorem generalizes \cite[Theorem 1.1]{LR1996} to the quasi-projective cases. Moreover,  we provide a more explicit description of the finite quotient group factoring through in\cite{LR1996} even when $X$ is projective. 

It is noteworthy that Deligne's theorem also has some analogue in terms of morphisms between character varieties. Let $X$ be a complex quasi-projective variety. The $\GL_{N}$-character variety of the topological fundamental group  $\pi_1(X)$ in characteristic zero, denoted by $M_{\rm B}(X,N)$, is an affine $\Q$-scheme of finite type such that its $\C$-points  $M_{\rm B}(X,N)(\C)$ is identified with the conjugate classes of semisimple representations $\pi_1(X)\to \GL_{N}(\C)$. In the case of characteristic $p>0$,  there exists also the corresponding character variety $\mxp$, which is an affine $\bF_p$-scheme of  finite type such that for any algebraically closed field $K$ with ${\rm char}\, K=p$, the $K$-points  $M_{\rm B}(X, N)(K)$ is identified with the conjugate classes of semisimple representations $\pi_1(X)\to \GL_{N}(K)$.   If  $f: Y\to X$ is a morphism between  smooth complex quasi-projective varieties,  
it induces  morphisms  $j_Y: M_{\mathrm{B}}(X, N) \to M_{\mathrm{B}}(Y, N)$ and $j_Y: \mxp\to \myp$.  
Our second result is as follows. 
\begin{thmx}[=\cref{thm:quasi-finiteness,thm:quasi-finiteness2}]\label{main2}
Let $f: Y\to X$ be a morphism between   smooth quasi-projective varieties, and $Z$ be the closure of $f(Y)$.    If ${\rm Im}[\pi_1(Z)\to \pi_1(X)]$  is a finite index subgroup of $\pi_1(X)$, then  both  $j_Y:M_{\rm B}(X,N)\to M_{\rm B}(Y,N)$  and $j_Y: \mxp\to \myp$ are quasi-finite.  
\end{thmx}
 In characteristic zero, this result generalizes \cite[Theorem 6.1]{Lasell1995} to the quasi-projective case. The case for a positive characteristic is new even in the projective setting.   We remark that our proof is quite different and much simpler than that in \cite{Lasell1995} (see \cref{rem:Lasell}).   Specifically, we avoid the use of the complex variation of Hodge structures, harmonic bundles, and moduli spaces of semistable Higgs bundles as constructed by  Simpson.   For the proofs of \cref{main,main2},  we mainly use the techniques of non-abelian Hodge theories in the non-archimedean cases recently developed in \cite{BDDM,CDY22,DY23,DY23b}.

\medspace

The paper is organized as follows. In \Cref{sec:pre}, we briefly review the character variety and properties of bounded representations and recall a reduction theorem in the quasi-projective case proved in \cite{CDY22}. In \Cref{sec:A}, we study the subscheme of the character varieties $M_\mathrm{B}(X, N)$ in characteristic 0 and    $M_\mathrm{B}(X, N)_{\bF_p}$ in characteristic $p$  consisting of all classes of representations whose pull-back via $f$ is trivial and prove \cref{main}.  \Cref{general} is devoted to the proof of \cref{main2}. 

\section{Technical preliminaries}\label{sec:pre}
\subsection{Character varieties}\label{Betti_moduli_space}
In this section, we briefly recall the definition and properties of the character varieties in any characteristic. For more details and proof of these results, we refer the readers to \cite{Ses77,LM85}.

Assume that $X$ is a complex quasi-projective variety and $\pi_1(X)$ is its topological fundamental group. Even though we can construct a variety of representations and so on for any finitely generated group, we restrict to the fundamental groups of algebraic varieties which is enough for our purpose. 

Let $N$ be a fixed positive integer. The homomorphisms from $\pi_1(X)$ into  $\mathrm{GL}_N$ is represented by an affine scheme $R_\mathrm{B}(X, N)_\bZ$ of finite type defined over $\mathbb{Z}$. For any commutative ring $A$, we have $R_\mathrm{B}(X, N)_\bZ(A)={\rm Hom}(\pi_1(X), {\rm GL}_N(A))$. Note that $\mathrm{GL}_N$ acts on $R_\mathrm{B}(X, N)_\bZ$ by the conjugation and such an action is algebraic. Let $\pi:R_\mathrm{B}(X, N)_\bZ\to M_\mathrm{B}(X, N)_\bZ$ be the GIT quotient. Note that $\pi $ is a surjective morphism between affine schemes of finite type defined over $\bZ$. We call $M_\mathrm{B}(X, N)_\bZ$  the  \emph{character  variety}  of $\pi_1(X)$ into ${\rm GL}_N$. 

Now assume that $A$ is a field $\K$ of characteristic zero.  Consider $M_\mathrm{B}(X, N):=M_\mathrm{B}(X, N)_\bZ\times_{\spec\, \bZ}\spec\, \Q$, which is an affine $\Q$-scheme of finite type.  For $\varrho\in R_\mathrm{B}(X, N)_{\bZ}(\K)$, we shall denote its image in $M_\mathrm{B}(X, N)(\K)$ by $[\varrho]:=\pi(\varrho)$. Notice that for an algebraically closed field $\K$ of characteristic zero, the points of $M_\mathrm{B}(X, N)(\K)$ parametrize conjugation classes of semisimple representations $\pi_1(X)\to \mathrm{GL}_N(\K)$. For any $\varrho\in R_\mathrm{B}(X, N)(\K)$, we denote its semi-simplification as $\varrho^{ss}$. Then two linear representations $\varrho_1,\varrho_2: \pi_1(X)\to \mathrm{GL}_N(\bar{\K})$ satisfies $[\varrho_1]=[\varrho_2]$ if and only if $\varrho_1^{ss}\sim\varrho_2^{ss}$, where $\sim$ means that they differ by a conjugation by an element in $\mathrm{GL}_N(\K)$.


The case of character variety in the positive characteristic is similar. Let $p $ be a prime number and write   $R_{\bF_p}:=R_{\rm B}(X,N)_{\bZ}  \times_{\spec\, \bZ}\spec\, \bF_p$. Note that the general linear group $\GL_{N,\bF_p}:=\GL_{N}\times_{\spec\, \bZ}\spec\, \bF_p$ over $\bF_p$ acts on $R_{\bF_p}$ by conjugation.  
Let $\mxp$ be the GIT quotient of $R_{\bF_p}$ by $\GL_{N,\bF_p}$. Then $M_{\rm B}(X,N)_{\bF_p}$  is an affine $\bF_p$-scheme of ﬁnite type.  For any algebraically closed field $\K$ of characteristic $p$, the $\K$-points  $M_{\rm B}(X,N)_{\bF_p}(\K)$ are identified with the conjugacy  classes of semisimple representations $\pi_1(X)\to \GL_N(\K)$. 

\subsection{Bounded representations}
We first recall the following definition of bounded subsets of affine schemes over a non-archimedean local field (see   \cite[Definitions 3.3]{DY23}).

\begin{definition}	
Let $\K$ be a non-archimedean local field.
 	\begin{thmlist}
 		 \item  Let $X$ be an affine $\K$-scheme of finite type. A subset $B\subset X(\K)$ is \textit{bounded} if for every $f\in \K[X]$, the set $\{\nu(f(b))\mid b\in B\}$ is bounded below, where $\nu:\K\to\R$ is the non-archimedean valuation of $\K$.  
 		 \item Let $\Gamma$ be a finitely generated group. A representation $\varrho:\Gamma\to {\rm GL}_N(\K)$ is \textit{bounded} if its image  $\varrho(\Gamma)$ is bounded. 
 	\end{thmlist} 
\end{definition}

Bounded subsets and compact subsets are closely related as follows (see   \cite[Fact 2.2.3]{KaPra23}).

\begin{lemma} \label{lem:bounded_and_compact}
Let $\K$ be a non-archimedean local field and $X$ be an affine $\K$-scheme of finite type. A closed subset $B$ is bounded if and only if $B$ is compact with respect to the analytic topology of $X(\K)$. If $f: X \to Y$ is a morphism of affine $\K$-schemes of finite type, then $f$ carries bounded subsets of $X(\K)$ to bounded subsets in $Y(\K)$.    \qed
\end{lemma}

The following property of bounded representations is essential for our proof.
\begin{lemma}[\protecting{\cite[Lemma 3.7]{DY23}}]\label{lem:DY23bounded}
	Let $\K$ be a non-archimedean local field. Let $x\in M_{\rm B}(X,N)(\K)$.  If $\{\varrho_i:\pi_1(X)\to {\rm GL}_N(\bar{\K})\}_{i=1,2}$ are two linear  representations such that $[\varrho_1]=[\varrho_2]=x\in M_{\rm B}(X, N)(\bar{\K})$, then $\varrho_1$ is bounded if and only if $\varrho_2$ is bounded. \qed
\end{lemma} 

In the case of characteristic zero, we need the following lemma due to Yamanoi.  
\begin{lemma}[\protecting{\cite[Lemma 4.2]{Yam10}}]\label{lem:Yamanoi}
Let $\K$ be a non-archimedean local field of characteristic zero. Let $R_0\subset R_{\mathrm{B}}(X, N)(\K)$ be the subset whose points are bounded representations. Let $M_0\subset M_{\mathrm{B}}(X, N)(\K)$ be the image of $R_0$ under the natural projection. Then $M_0$ is compact with respect to the analytic topology. \qed
\end{lemma}

\subsection{A reduction theorem}
Based on the previous work \cite{BDDM} on the extension of Gromov-Schoen theory \cite{GS92} to quasi-projective varieties,   Cadorel, Yamanoi, and the first author \cite{CDY22} established the following theorem for the representation of fundamental groups of quasi-projective varieties into algebraic groups defined over non-archimedean local fields. 
\begin{theorem}[\protecting{\cite[Theorem 0.10]{CDY22}}] \label{thm:KE}
Let $X$ be a complex smooth quasi-projective variety, and let $\varrho:\pi_1(X)\to {\rm GL}_N(\K)$ be a reductive representation where $\K$ is a non-archimedean local field.  Then 
	 for any connected Zariski closed  subset $T$ of $X$, the following properties are equivalent:
	 \begin{enumerate}[label={\rm (\alph*)}]
		\item \label{item bounded} the image $\varrho({\rm Im}[\pi_1(T)\to \pi_1(X)])$ is a bounded subgroup of $\mathrm{GL}_N(\K)$.
		\item \label{item normalization} For every irreducible component $T_o$ of $T$, the image $\varrho({\rm Im}[\pi_1(T_o^{\rm norm})\to \pi_1(X)])$ is a bounded subgroup of $\mathrm{GL}_N(\K)$. Here $T_o^{\rm norm}$ means the normalization of $T_o$.\qed
	\end{enumerate} 
\end{theorem}


\section{Proof of \texorpdfstring{\cref{main}}{Theorem A}}\label{sec:A}
\subsection{Properties of    character varieties in characteristic zero}
  Let $f: Y\to X$ be a morphism between connected smooth complex quasi-projective varieties and $Z$ be the Zariski closure of the image of $f$ in $X$. Write $\iota: Z\to X$ as the natural inclusion. There are natural morphisms between character varieties and representation schemes of $X, Y$, and $Z$ induced by $f$ and $\iota$:
  \begin{equation}
  	\begin{tikzcd} \label{main_diagram}
  		R_\mathrm{B}(X,N) \arrow{r}{\pi_X} \arrow{d}{\iota^*}\arrow[swap,bend right=60]{dd}{f^*} & M_\mathrm{B}(X,N) \arrow{d}{j_Z} \arrow[bend left=60]{dd}{j_Y} & M_Y \arrow[hook]{l}\arrow{dd}{j_Y}\\
  		R_\mathrm{B}(Z,N) \arrow{r}{\pi}\arrow{d}  & M_\mathrm{B}(Z,N) \arrow{d}{}& \\
  		R_\mathrm{B}(Y,N) \arrow{r}{\pi_Y} & M_\mathrm{B}(Y,N) & \{[1]\}\arrow[hook]{l}\\
  	\end{tikzcd} 
  \end{equation}
Here we define
\begin{equation} \label{working_subscheme}
M_Y:=j_Y^{-1}([1])
\end{equation}
where $[1]$ stands for the trivial class in $M_{\rm B}(Y,N)$.  Then $M_Y$ is a closed subscheme of $M_{\rm B}(X, N)$.

The following result is a consequence of the properties of harmonic maps to symmetric spaces or to Euclidean buildings. For this purpose, the representation being reductive is enough.

\begin{proposition} \label{prop:compact&bounded}
Let $f: Y\to X$ be a  morphism between smooth quasi-projective varieties and $Z$ be the Zariski closure of $f(Y)$.
\begin{enumerate}[label=\rm (\alph*)]
\item  \label{item:unitary}	If $\varrho: \pi_1(X)\to \mathrm{GL}_N(\C)$ is a semisimple representation such that $[\varrho]\in{M_Y}(\C)$, then $\varrho(\mathrm{Im}[\pi_1(Z)\to \pi_1(X)])\subset g{\rm U}_N(\C)g^{-1}$, where ${\rm U}_N(\C)$ is the unitary group of degree $N$, and $g\in {\rm GL}_N(\C)$. 
\item  \label{item:bounded}	If $\varrho: \pi_1(X)\to \mathrm{GL}_N( \K)$ is a semisimple representation  such that $[\varrho]\in{M_Y}(\K)$ where $\K$ is a non-archimedean local field, then $\varrho(\mathrm{Im}[\pi_1(Z)\to \pi_1(X)])$ is bounded. 
\end{enumerate}
\end{proposition}
  \begin{proof}
 We might assume that $f$ is proper. \\
\noindent {\it Proof of \Cref{item:unitary}.} Write $S:=\mathrm{GL}_N(\C)/{\rm U}_N(\C)$.  Let $\pi_\varrho:\widetilde{X}_\varrho\to X$ be the Galois covering corresponding to $\ker \varrho$.  By a theorem of Mochizuki \cite{Moc07b}, there exists a $\varrho$-equivariant harmonic mapping $u:\widetilde{X}_\varrho\to S$ corresponding to a tame pure imaginary harmonic bundle $(E,\theta,h)$.  By \cref{working_subscheme},  we have $f^*\varrho=1$, and thus there exists a lift  $\widetilde{f}: Y\to\widetilde{X}_\varrho$  of $f$.  Then $u\circ\widetilde{f}:Y\to S$ is a $f^*\varrho$-equivariant harmonic mapping corresponding to a tame pure imaginary bundle $(f^*E,f^*\theta,f^*h)$. Since $f^*\varrho=1$, by the unicity theorem of harmonic maps in \cite{Moc07b}, it follows that $(f^*E,f^*\theta,f^*h)=( \oplus^N\mathcal{O}_Y,0,h_0)$, where $h_0$ is the canonical metric for $\oplus^N\mathcal{O}_Y$ whose curvature is zero. Therefore, $u\circ\widetilde{f}$ is a constant map.

\[
\begin{tikzcd}
&\widetilde{X}_{\varrho}\arrow{r}{u}\arrow{d}{\pi_{\varrho}}& S\\
Y\arrow[swap]{r}{f}\arrow{ru}{\widetilde{f}} & X&
\end{tikzcd}
\]
 Since $f$ is proper,  $\tilde{f}$ is also proper and thus $\widetilde{f}(Y)$ is an irreducible analytic subvariety of $\widetilde{X}_{\varrho}$. Therefore, $\widetilde{f}(Y)$ is an irreducible component of $\pi_\varrho^{-1}(f(Y))$. If we choose different lifts $\widetilde{f}$ of $f$, $\widetilde{f}(Y)$ can be mapped surjectively onto any given irreducible component of $\pi_\varrho^{-1}(f(Y))$. Since $u\circ\widetilde{f}(Y)$ is a point, it follows that for any connected component $W$ of $\pi_\varrho^{-1}(f(Y))$, the image $u(W)$ is a point $P\in S$. This implies that $\varrho({\rm Im}[\pi_1(Z)\to \pi_1(X)])$ fixes $P$. Let $g\in {\rm GL}_N(\C)$ such that $g{\rm U}_N(\C)=P$.   Hence $\varrho({\rm Im}[\pi_1(Z)\to \pi_1(X)])\subset g{\rm U}_N(\C)g^{-1}$.  
 
 \medspace
 
 \noindent {\it Proof of \Cref{item:bounded}.} By \eqref{working_subscheme}, we know that  $f^*\varrho=1$, which is thus bounded. It follows from \cref{thm:KE} that $\varrho({\rm Im}[\pi_1(Z)\to \pi_1(X)])$ is bounded. 
\end{proof}

\begin{proposition}\label {prop:finite_image_for_Z}
Let $\varrho:\pi_1(X)\to {\rm GL}_N(k)$ be a reductive representation where $k$ is an algebraic number field. If $\varrho\in M_Y(k)$, then $\varrho({\rm Im}[\pi_1(Z)\to \pi_1(X)])$ is finite.   
\end{proposition}
\begin{proof}
Let ${\rm Ar}(k)$ be the set of archimedean places of $k$. For any $w\in {\rm Ar}(k)$, set $\varrho_w=w\circ \varrho:\pi_1(X)\to {\rm GL}_N(\C)$. Note that $M_Y$ is defined over $\Q$ as $[1]\in M_{\rm B}(Y,N)(\Q)$.   It follows that  $\varrho_w\in M_Y(\C)$.     By \cref{prop:compact&bounded}, $\varrho_w({\rm Im}[\pi_1(Z)\to \pi_1(X)])\subset g_{w}{\rm U}_N(\C)g_{w}^{-1}$, for some $g_w\in {\rm GL}(N, \C)$. 
	
For any non-archimedean place $\nu$ of $k$, denote by $k_\nu$ its non-archimedean completion. Let $\varrho_{\nu}:\pi_1(X)\to {\rm GL}(N,k_\nu)$ be representation induced by $\varrho$.   By \cref{prop:compact&bounded}, $\varrho_\nu({\rm Im}[\pi_1(Z)\to \pi_1(X)])$ is bounded for any $\nu$. It follows that $\varrho({\rm Im}[\pi_1(Z)\to \pi_1(X)])\subset {\rm GL}_N(O_{k})$, where $O_{k}$ is the ring of integers. Note that ${\rm GL}_N(O_{k})\subset \prod_{w\in {\rm Ar}(k)}{\rm GL}_N(\C)$ is discrete. It follows that for the representation $\sigma=\oplus_{w\in {\rm Ar}(k)}\varrho_w: \pi_1(X)\to \prod_{w\in {\rm Ar}(k)}{\rm GL}_N(\C)$, we have $\sigma({\rm Im}[\pi_1(Z)\to \pi_1(X)])\subset {\rm GL}_N(O_{k})\cap \prod_{w\in {\rm Ar}(k)}g_{w}{\rm U}_N(\C)g_{w}^{-1}$, which is finite. Therefore, $\varrho({\rm Im}[\pi_1(Z)\to \pi_1(X)])$  is finite. 
\end{proof}


\subsection{Proof of  \texorpdfstring{\cref{main}}{Theorem A} in characteristic zero} 

The following proposition is a variant of   \cite[Claim 3.10]{DY23}.
\begin{proposition} \label{prop:zero_dimensional}
The subset $j_Z(M_Y(\overline{\Q}))$ is $0$-dimensional and consists of finitely many points, say $\{[\tau_i:\pi_1(Z)\to {\rm GL}_N(\overline{\Q})]\}_{i=1,\ldots,\ell}$. Furthermore, the number $\ell$ is no more than the number of the geometrically connected components of $M_Y$.
\end{proposition}

\begin{proof}
We only need to show that for any two semisimple representations $\varrho_1,\varrho_2:\pi_1(X)\to \mathrm{GL}_N(\overline{\Q})$ such that $[\varrho_1], [\varrho_2]$ are in the same geometrically connected component of $M_Y(\overline{\Q})$, we have $j_Z([\varrho_1])=[\iota^*(\varrho_1)]=[\iota^*(\varrho_2)]=j_Z([\varrho_2])$. 

Let $\mathfrak{R}:=\pi_X^{-1}(M_Y)\subset R_\mathrm{B}(X,N)$. Then $\varrho_1, \varrho_2\in \mathfrak{R}(\overline{\Q})$. Since $\pi_1(Z), \pi_1(X)$ are finitely generated,  $M_\mathrm{B}(X, N), M_\mathrm{B}(Z, N)$ as affine schemes of finite type defined over $\Q$ and $\mathfrak{R}$ is a Zariski closed subset   defined over $\overline{\Q}$.   

We first show that the conclusion holds if $\varrho_1, \varrho_2$ are in the same geometrically irreducible component of $\mathfrak{R}(\overline{\Q})$. We prove this by contradiction. Assume that $[\iota^*(\varrho_1)]\neq [\iota^*(\varrho_2)]$, then there exists a $\overline{\Q}$-morphism $\psi: M_\mathrm{B}(Z, N) \to \A^1$ such that $\psi([\iota^*(\varrho_1)])\neq \psi([\iota^*(\varrho_2)])$. We can find a   closed irreducible curve $C\subset \mathfrak{R}$ containing both $\varrho_1$ and $\varrho_2$.  Then $\psi\circ \pi \circ \iota^*|_{C}=\psi\circ j_Z\circ \pi_X|_{C}: {C}\to \A^1$ is non-constant and thus generically finite. We can find an open subset ${U}\subset \A^1$ over which the above map is finite.  Let $\K_0$ be a finite extension of $\Q$ such that $C$ is defined over $\K_0$ and $\psi\circ \pi \circ \iota^*|_{C}$   is a $\K_0$-morphism.  Let $\nu$ be a non-archimedean place of $\K_0$ and $\K$ be the completion of $\K_0$ with respect to $\nu$. Then $\K$ is a non-archimedean local field of characteristic zero.

Take $x\in U(\K)$ and $\varrho\in C(\overline{\K})$ over $x$. Then $\varrho$ is defined over some finite extension of $\K$ whose degree is bounded by the degree $\psi\circ \pi \circ \iota^*|_{C}$. There are finitely many extensions and we can assume that all points over $U(\K)$ are contained in $C(\mathbb{L})$, where $\mathbb{L}$ is a finite extension of $\K$.  Since $U(\K)$ is unbounded, we have $\psi\circ \pi \circ \iota^*(C(\mathbb{L}))\subset \A^1(\mathbb{L})$ is unbounded.

Take $R_0$ to be the set of all bounded representations in $R_\mathrm{B}(Z, N)(\mathbb{L})$, then $M_0:=\pi(R_0)$ is compact in $R_\mathrm{B}(Z, N)(\mathbb{L})$ with respect to analytic topology by \cref{lem:Yamanoi}. By \cref{lem:bounded_and_compact}, $M_0$ is bounded and $\psi(M_0)$ is also bounded in $\A^1(\mathbb{L})$. Thus there exists $\varrho\in C(\mathbb{L})$ such that $\pi\circ \iota^*(\varrho)=[\iota^*(\varrho)]\notin M_0$. It follows that $\iota^*\varrho:\pi_1(Z)\to {\rm GL}_N(\mathbb{L})$ is unbounded. On one hand, by \Cref{lem:DY23bounded}, $(\iota^*\varrho)^{ss}:\pi_1(Z)\to {\rm GL}_N(\overline{\mathbb{L}})$  is also unbounded. On the other hand, note that $[\varrho^{ss}]=[\varrho]\in M_Y(\overline{\mathbb{L}})$. By the definition of $M_Y$, we have $f^*\varrho^{ss}:\pi_1(Y)\to {\rm GL}_N(\overline{\mathbb{L}})$ is trivial.   It follows   \Cref{prop:compact&bounded} that $\iota^*(\varrho^{ss}):\pi_1(Z)\to  {\rm GL}_N(\overline{\mathbb{L}})$ is bounded.  Note that $[\iota^*(\varrho^{ss})]=[\iota^*\varrho]$. By \Cref{lem:DY23bounded} again, $(\iota^*\varrho)^{ss}$ is also bounded. This is a contradiction, and thus we must have $[\iota^*(\varrho_1)]= [\iota^*(\varrho_2)]$ when $\varrho_1, \varrho_2$ are in the same geometrically irreducible component of $\mathfrak{R}$.


Let $M'$  be a   geometrically connected component of $M_Y$. Consider a geometrically irreducible component $M''$ of $M'$. We can choose a geometrically irreducible component $W$ of $\pi_X^{-1}(M'')$ such that $\pi_X(W)$ is dense in $M''$. It follows that $W$ is an irreducible component of $\mathfrak{R}$. By the above argument, we know that $j_Z \circ \pi_X(W)$ is a point in $M_{\rm B}(Z, N)$. Thus, $j_Z(M'')$ is also a point in $M_{\rm B}(Z,N)$. 
Consequently, $j_Z(M')$ is a point in $M_{\rm B}(Z,N)$.

Let $M_1,\ldots, M_k$ be all geometrically connected components of $M_Y$, which are all defined over $\overline{\Q}$. We then can take semisimple representations $\{\varrho_i:\pi_1(X)\to {\rm GL}_N(\overline{\Q})\}_{j=1,\ldots,k}$  such that $[\varrho_i]\in M_i(\overline{\Q})$. Then the image $j_Z(M_Y(\overline{\Q}))=\{j_Z([\varrho_i]) \}_{i=1,\ldots, k}=:\{[\tau_i] \}_{i=1,\ldots, \ell}$ for some $\ell\leqslant k$.  The proposition is proved.  
\end{proof}

Now we can prove \cref{main} in characteristic zero.

\begin{theorem}\label{thm:trivial}
Let $f: Y\to X$ be a morphism between connected smooth quasi-projective varieties and $Z$ be the closure of the image of $f$. There exists finitely many semisimple representations $\{\tau_i:\pi_1(Z)\to {\rm GL}_N(\overline{\Q})\}_{i=1,\ldots,\ell}$ such that  
	\begin{thmlist}
		\item for any $i$, $\tau_i(\pi_1(Z))$ is a finite group;
	\item if $\varrho:\pi_1(X)\to {\rm GL}_N(K)$ is a semisimple representation with K a  field of characteristic zero such that $f^*\varrho=1$, then $\iota^*\varrho:\pi_1(Z)\to {\rm GL}_N(\C)$ is conjugate to some $\tau_i$. Here $\iota: Z\hookrightarrow X$ is the natural inclusion.  
	\end{thmlist}
\end{theorem}
\begin{proof}
By \cref{prop:zero_dimensional}, $j_Z(M_Y(\overline{\Q}))$ is zero dimensional.  Let $\{\tau_i:\pi_1(Z)\to {\rm GL}_N(\overline{\Q})\}_{i=1,\ldots,\ell}$ be semisimple representations such that $\{[\tau_i]\}_{i=1,\ldots,\ell}$ is the image $j_Z(M_Y(\overline{\Q}))$.  Then there exist semisimple representations $\{\varrho_i:\pi_1(X)\to {\rm GL}_N(\overline{\Q})\}_{i=1,\ldots,\ell}$ such that $[\varrho_i]\in M_Y(\overline{\Q})$ and $[\iota^*\varrho_i]=j_Z([\varrho_i])=[\tau_i]$.  By \cref{prop:finite_image_for_Z},   $\iota^*\varrho_i(\pi_1(Z))$ is a finite group. Hence $\iota^*\varrho_i$ is a semisimple representation.  It follows that $\iota^*\varrho_i$ is conjugate to $\tau_i$. Hence $\tau_i(\pi_1(Z))$ is finite.   

Since $\pi_1(X)$ is finitely presented,  there exists a subfield $k\subset K$ with ${\rm tr.deg.}(k/\Q)<\infty$ such that   $\varrho(\pi_1(X))\subset \GL_{N}(k)$. Then there exists an embedding $k\hookrightarrow \C$ and we may think of $\varrho$ as a complex linear representation.       Therefore, $[\iota^*\varrho]=j_Z([\varrho])=[\tau_i]$ for some $i$ by \cref{prop:zero_dimensional}.  By \cref{prop:compact&bounded}, we know that $\iota^*\varrho(\pi_1(Z))\subset {\rm U}_N(\C)$ up to a congujation. Note that any unitary representation is semisimple. It follows that $\iota^*\varrho:\pi_1(Z)\to {\rm GL}_N(\C)$ is a semisimple representation. Hence $\iota^*\varrho$ is conjugate to $\tau_i$. The theorem is proved.    
\end{proof}

\begin{remark}
\begin{thmlist}
	\item  	The above theorem is a generalization of  \cite[Theorem 1.1]{LR1996} in characteristic $0$ to the quasi-projective case, along with a more explicit description of the finite set $\Delta_N$ ($N$ represents $n$ there) they defined. 
	\item  A weaker result is proven in \cite[Theorem A]{DY23}: if $\varrho:\pi_1(X)\to {\rm GL}_N(\C)$ is a semisimple representation such that $f^*\varrho=1$, then $\varrho({\rm Im}[\pi_1(Z)\to \pi_1(X)])$ is finite.  
\end{thmlist}
\end{remark}

\subsection{Proof of  \texorpdfstring{\cref{main}}{Theorem A} in positive characteristic}
In this subsection, let $f: Y\to X$ be a morphism between smooth complex quasi-projective varieties and let $Z$ be the closure of $f(Y)$. It induces a morphism $j_Y:\mxp\to \myp$ between affine $\bF_p$-schemes of finite type. The natural inclusion $\iota: Z\hookrightarrow X$ induces $j_Z:\mxp\to \mzp$ which is also a morphism between affine $\bF_p$-schemes of finite type.  Let $M_{Y,p}:=j_Y^{-1}([1])$, which is a Zariski closed subset of $\mxp$ defined over $\bF_p$.

We first recall the following lemma in \cite{DY23b}.  
\begin{lemma}\label{lem:finite group}
	Let $\K$ be an algebraically closed field of characteristic $p>0$ and let $\Gamma$ be a finitely generated group. Let $\varrho:\Gamma\to G(\K)$ be a representation such that its semisimplification is conjugate to some $\tau:\Gamma\to \GL_{N}(\overline{\bF_p})$. Then $\varrho(\Gamma)$ is finite.
\end{lemma} 
Since the proof is short, we recall it for the sake of completeness. 
\begin{proof}[Proof of \cref{lem:finite group}]
	Since the image of $\tau$ is finite, we can replace $\Gamma$ with a finite index subgroup such that $\tau(\Gamma)$ is trivial. Hence, we can assume that the semisimplification of $\varrho$ is trivial. Therefore, some conjugation $\sigma$ of $\varrho$ has image in ${\rm U}_N(\K)$
	of all upper-triangular matrices in $\GL_{N}(\K)$
	with 1's on the main diagonal.  Note that  ${\rm U}_N(\K)$  admits a central normal series whose successive
	quotients are isomorphic to $\mathbb{G}_{a,\K}$. It follows that $\sigma(\Gamma)$ admits a   central normal series whose successive
	quotients are finitely generated subgroups of $\mathbb{G}_{a,\K}$, which are finite groups.  It follows that $\sigma(\Gamma)$  is finite. The lemma is proved. 
\end{proof}
\begin{lemma}\label{lem:finite}
	The subset $j_Z(M_{Y,p})$ is $0$-dimensional and consists of finitely many points, say $\{[\tau_i:\pi_1(Z)\to {\rm GL}_N(\overline{\bF_p})]\}_{i=1,\ldots,\ell}$. Furthermore, the number $\ell$ is no more than the number of the geometrically connected components of $M_{Y,p}$. 
\end{lemma}
\begin{proof}
Assume by contradiction that $j_Z(M_{Y,p})$ is positive dimensional.  Write $R_{\bF_p} $ for $R_{\rm B}(X,N)_{\bZ}  \times_{\spec\, \bZ}\spec\, \bF_p$.  Since the morphism $\pi_p: R_{\bF_p}\to \mxp$ is surjective between affine $\bF_p$-schemes of finite type, we can find an affine irreducible curve $C_o\subset \pi_p^{-1}(M_{Y,p})$ defined over $\overline{\bF_p}$ such that $j_Z\circ \pi_p(C_o)$  is positive dimensional. Let $\overline{C}$ be a compactification of the normalization $C$ of $C_o$, and let $\{P_1,\ldots,P_\ell\}= \overline{C}\setminus C$.  Let $q=p^n$ such that $\overline{C}$ is defined over $\bF_q$ and $P_i\in R_{\bF_p}(\bF_q)$ for each $i$. 

By the universal property of the representation scheme $R$, $C$ gives rise to a representation $\varrho_C:\pi_1(X)\to \GL_N(\bF_q[C])$, where $\bF_q[C]$ is the coordinate ring of $C$. Consider the discrete valuation $\nu_i:\bF_q(C)\to \bZ$ defined by $P_i$, where $\bF_q(C)$ is the function field of $C$. Let $\widehat{\bF_q(C)}_{\nu_i}$  be the completion of $F_{q}(C)$ with respect to $\nu_i$. Then we have $\big(\widehat{\bF_q(C)}_{\nu_i},\nu_i\big)\simeq \big(\bF_q((t)),\nu\big)$, where $ \big(\bF_q((t)),\nu\big)$ is the formal Laurent field of $\bF_q$ with the valuation $\nu$ defined by  $\nu(\sum_{i=m}^{+\infty}a_it^i)=\min \{i\mid a_i\neq 0\}$.  
  Let $\varrho_i:\pi_1(X)\to \GL_N(\bF_q((t)))$ be the extension of $\varrho_C$ with respect to $\widehat{\bF_q(C)}_{\nu_i}$. 
  \begin{claim}\label{lem:simple2}
  	There exists some $i$ such that $\iota^*\varrho_i:\pi_1(Z)\to \GL_N(\bF_q((t)))$ is unbounded.
  \end{claim}
  \begin{proof}
  	This claim is proved in \cite{BDDM,DY23b} and we recall it here for the sake of completeness. Assume for the sake of contradiction that $\iota^*\varrho_i$ is bounded for each $i$. Then  after we replace $\iota^*\varrho_i$ by some conjugation, we have $\iota^*\varrho_i(\pi_1(Z))\subset \GL_{N}(\bF_q[[t]])$.  For any matrix $A\in \GL_N(K)$,  we denote by $\chi(A)=T^N+\sigma_1(A)T^{N-1}+\cdots+\sigma_N(A)$ its characteristic polynomial.    Then $\sigma_j(\iota^*\varrho_C(\gamma))\in \bF_q[C]$ for each $\gamma\in \pi_1(Z)$.  
  	Since we have assumed that $\iota^*\varrho_i(\pi_1(Z))\subset \GL_{N}(\bF_q[[t]])$ for each $i$, it follows that $\sigma_{j}(\iota^*\varrho_i(\gamma))\in \bF_q[[t]]$ for each $i$.  Therefore, by the definition of $\varrho_i$, $\nu_i\big(\sigma_j(\iota^*\varrho_C(\gamma))\big)\geq 0$ for each $i$.  It follows that $\sigma_j(\iota^*\varrho_C(\gamma))$  extends to a regular function on $\overline{C}$, which is thus constant.  Since  the conjugate classes of semisimple representations are determined by their characteristic polynomials, it follows that $j_Z\circ\pi_p(C_o)$ is a point, which contradicts to our assumption at the beginning. 	Hence there exists some $i$ such that $\iota^*\varrho_i:\pi_1(Z)\to \GL_N(\bF_q((t)))$ is unbounded.
  \end{proof}
  For each $i$, note that $[\varrho_i]\in M_{Y,p}(\bF_q((t)))$. It follows that $[f^*\varrho_i]=1$.  By \cref{lem:DY23bounded}, $f^*\varrho_i$ is bounded. Thanks to \cref{thm:KE}, $\iota^*\varrho_i$ is also bounded. This contradicts to \cref{lem:simple2}. 
\end{proof}

\begin{theorem}\label{thm:trivial2}
	Let $f: Y\to X$ be a morphism between connected smooth quasi-projective varieties and $Z$ be the closure of the image of $f$. For any prime number $p>0$,  there exists finitely many semisimple representations $\{\tau_i:\pi_1(Z)\to {\rm GL}_N(\overline{{\bF_p}})\}_{i=1,\ldots,\ell}$ such that     if $\varrho:\pi_1(X)\to {\rm GL}_N(\K)$ is a linear  representation where ${\rm char}\, \K=p$ such that $[f^*\varrho]=1$, then the semisimplification of $\iota^*\varrho:\pi_1(Z)\to {\rm GL}_N(\K)$ is conjugate to some $\tau_i$. Here $\iota:Z\hookrightarrow X$ is the natural inclusion.    In particular, $\iota^*\varrho(\pi_1(Z))$  is finite.
\end{theorem}
\begin{proof}
	Without loss of generality, we can assume that $\K$ is algebraically closed. Note that $[\varrho]\in M_{Y,p}(\K)$ as $[f^*\varrho]=1$.  By \cref{lem:finite}, $j_Z([\varrho])=[\tau_i]$ for some semisimple $\tau_i:\pi_1(Z)\to {\rm GL}_N(\overline{\bF_p})$ therein.  Hence $\tau_i$ is conjugate to the semisimplification of  $\iota^*\varrho$.   By \cref{lem:finite group}, we conclude that $\iota^*\varrho(\pi_1(Z))$  is finite. 
\end{proof}

\begin{remark}
	In 	   \cite[Theorem A]{DY23b}, the first author and Yamanoi constructed the Shafarevich morphism for any   representation  $\varrho:\pi_1(X)\to {\rm GL}_N(K)$ where $X$ is a quasi-projective normal variety and $K$ is a field of positive characteristic. As a byproduct,   a  similar result compared with \cref{thm:trivial2} is proved: if $\varrho:\pi_1(X)\to {\rm GL}_N(K)$ is a  linear representation such that $f^*\varrho(\pi_1(Y))$ has finite image, then $\iota^*\varrho(\pi_1(Z))$  is  also finite. 
\end{remark}

\section{Proof of \texorpdfstring{\cref{main2}}{Theorem B}}\label{general} 
\subsection{Character variety in characteristic zero}
\begin{theorem}\label{thm:quasi-finiteness}
Let $f: Y\to X$ be a morphism between two connected smooth quasi-projective varieties and $Z$ be the Zariski closure of the image of $f$. If  ${\rm Im}[\pi_1(Z)\to \pi_1(X)]$  is a finite index subgroup of $\pi_1(X)$,  then $j_Y:M_{\rm B}(X,N)\to M_{\rm B}(Y,N)$ is a quasi-finite morphism.  
\end{theorem}
The proof of \cref{thm:quasi-finiteness} is similar to that of \cref{prop:zero_dimensional}.
\begin{proof}
Since $j_Y$ is a $\mathbb{Q}$-morphism of affine schemes of finite type, it suffices to prove that for any $\overline{\mathbb{Q}}$-point $x$ of $M_{\rm B}(Y,N)$, $j_Y^{-1}(x)$ is a finite set of $\overline{\mathbb{Q}}$-point. We assume by contradiction that there exists a semisimple representation $\sigma:\pi_1(Y)\to\mathrm{GL}_N(\overline{\mathbb{Q}})$ such that $j_Y^{-1}([\sigma])$ is not finite. 

Set $M:=j_Y^{-1}([\sigma])$. It is a positive dimensional closed subscheme of $M_{\rm B}(X, N)$ defined over $\overline{\mathbb{Q}}$. Since $\pi_1(Y)$ is finitely generated, there exists a number field $k$ such that $\sigma:\pi_1(Y)\to \mathrm{GL}_N(k)$. Moreover, there exists a non-archimedean place $\nu$ such that $\sigma:\pi_1(Y)\to \mathrm{GL}_N(k_{\nu})$ is bounded,  where $k_{\nu}$ is the completion of $k$ with respect to  $\nu$. 
 
 Let $\mathfrak{R}:=\pi_X^{-1}(M)\subset R_\mathrm{B}(X,N)$, where $\pi_X: R_\mathrm{B}(X, N)\to M_{\rm B}(X, N)$  is the GIT quotient. Since $M$ is a positive dimensional affine scheme defined over $\overline{\Q}$, there exists a $\overline{\Q}$-morphism $\psi: M\to \mathbb{A}^1$ whose image is Zariski dense. Since $\pi_X$ is surjective, we can find a closed irreducible curve $C\subset \mathfrak{R}$  such that $\psi\circ \pi_X|_{C}: {C}\to \A^1$ is   generically finite. 

Let $\K$ be a finite extension of $k_{\nu}$ such that $C$ is defined over $\K$ and $\psi\circ \pi_X|_{C}$ is a morphism of $\K$-schemes. We can find an open subset ${U}\subset \A^1$ over which the above map is finite.  Take $x\in U(\K)$ and $\varrho\in C(\overline{\K})$ over $x$, then $\varrho$ is defined over some finite extension of $\K$ whose degree is controlled above by the degree of $\psi\circ \pi_X|_{C}$. There are finitely many extensions and we can assume that all points over $U(\K)$ are contained in $C(\mathbb{L})$, where $\mathbb{L}$ is a finite extension of $\K$. Since $U(\K)$ is unbounded, we have $\psi\circ \pi_X(C(\mathbb{L}))\subset \A^1(\mathbb{L})$ is unbounded.
 
Take $R_0$ to be the set of all bounded representations in $R_\mathrm{B}(X, N)(\mathbb{L})$, then $M_0:=\pi_X(R_0)$ is compact in $M_\mathrm{B}(X, N)(\mathbb{L})$ with respect to analytic topology by  \cref{lem:Yamanoi} and  $\psi(M_0)$ is also bounded in $\A^1(\mathbb{L})$ by \cref{lem:bounded_and_compact}. It follows that there exists $\varrho\in C(\mathbb{L})$ such that $\pi_X (\varrho)=[\varrho]\notin M_0$. Let $\tau=\varrho^{\mathrm{ss}}:\pi_1(X)\to {\rm GL}_N(\overline{\mathbb{L}})$ be the semisimplification of $\varrho$.  By \Cref{lem:DY23bounded}, $\tau:\pi_1(X)\to {\rm GL}_N(\overline{\mathbb{L}})$  is also unbounded. Note that $[\tau]=[\varrho]\in M_{\rm B}(X,N)(\overline{\mathbb{L}})$. By the definition of $C$, we have $[f^*\tau]=j_Y([\tau])=[\sigma]$ as points in $M_{\rm B}(Y,N)(\overline{\mathbb{L}})$. Note that $f^*\tau$ is semisimple by \cite[Theorem 25.30]{Moc07b}, it follows that $f^*\tau$ is conjugate to $\sigma:\pi_1(Y)\to {\rm GL}_N(\overline{\mathbb{L}})$.

Let $\mathbb{L}'$ be a finite extension of $\mathbb{L}$ such that $\tau:\pi_1(X)\to \mathrm{GL}_N(\mathbb{L}')$. We think of $\sigma$ as a representation $\sigma:\pi_1(Y)\to \mathrm{GL}_N(\mathbb{L}')$, which is bounded as $\mathbb{L}'$ is a finite extension of $k_\nu$.    Hence $f^*\tau:\pi_1(Y)\to {\rm GL}_N(\mathbb{L}') $ is also bounded.  This implies that $\tau(\mathrm{Im}[\pi_1(Z)\to \pi_1(X)])$ is bounded thanks to \cref{thm:KE}. Since we assume that   $\mathrm{Im}[\pi_1(Z)\to \pi_1(X)]$ is a finite index subgroup of $\pi_1(X)$, it follows that $\tau( \pi_1(X))$ is also  bounded. This contradicts that   $\tau$ is unbounded. Therefore, $j_Y^{-1}([\sigma])$ is a finite set. Since $[\sigma]$ is an arbitrary point in $M_{\rm B}(Y,N)(\overline{\Q})$, it follows that $j_Y$ is quasi-finite.  We proved the theorem. 
\end{proof}
\begin{remark}\label{rem:Lasell}
The above theorem generalizes   \cite[Theorem 6.1]{Lasell1995}    to the quasi-projective setting.   
We remark that the original proof by Lasell in the projective case is quite involved and is substantially built on Deligne's mixed Hodge theory,  Simpson's work of harmonic bundles \cite{Sim92} and his construction of moduli space of semistable Higgs bundles   $M_{\mathrm{Dol}}$ in \cite{Sim94, Sim94b}. Notably, the construction of  $M_{\mathrm{Dol}}$ in the quasi-projective cases has not been established yet.   Therefore, we cannot apply the same method by Lasell to prove  \cref{thm:quasi-finiteness}.   
\end{remark}

\subsection{Character variety in positive characteristic}
 \begin{theorem}\label{thm:quasi-finiteness2}
 	Let $f: Y\to X$ be a morphism between connected smooth quasi-projective varieties and $Z$ be the closure of the image of $f$. If  ${\rm Im}[\pi_1(Z)\to \pi_1(X)]$  is a finite index subgroup of $\pi_1(X)$, then $j_Y:\mxp\to \myp$ is a quasi-finite morphism.  
 \end{theorem}
  \begin{proof}
Since $j_Y:\mxp\to \myp$ is a morphism of affine $\bF_p$-schemes of finite type,  	it suffices to prove that $j_Y^{-1}([\sigma])$ is zero-dimensional for any  closed point $[\sigma]\in \myp(\overline{{\bF}_p})$.  Assume, for the sake of contradiction, that $j_Y^{-1}([\sigma])$ is positive dimensional for some semisimple $\sigma:\pi_1(Y)\to \GL_{N}(\overline{\bF}_p)$.  Write $R_{\bF_p} $ for $R_{\rm B}(X,N)_{\bZ}  \times_{\spec\, \bZ}\spec\, \bF_p$.   Since the GIT quotient $\pi_p: R_{\bF_p}\to \mxp$ is a surjective morphism between affine $\bF_p$-schemes of finite type, we can find an irreducible affine curve $C_o\subset (j_Y\circ\pi_p)^{-1}([\sigma])$ defined over $\overline{\bF_p}$ such that $\pi_p(C_o)$  is positive dimensional. Let $\overline{C}$ be a compactification of the normalization $C$ of $C_o$, and let $\{P_1,\ldots,P_\ell\}= \overline{C}\setminus C$.  Let $q=p^n$ such that $\overline{C}$ is defined over $\bF_q$ and $P_i\in R_{\bF_p}(\bF_q)$ for each $i$. 
 	
By the universal property of the representation scheme $R$, $C$ gives rise to a representation $\varrho_C:\pi_1(X)\to \GL_N(\bF_q[C])$, where $\bF_q[C]$ is the coordinate ring of $C$. Consider the discrete valuation $\nu_i:\bF_q(C)\to \bZ$ defined by $P_i$, where $\bF_q(C)$ is the function field of $C$. Let $\widehat{\bF_q(C)}_{\nu_i}$  be the completion of $F_{q}(C)$ with respect to $\nu_i$. Then we have $\big(\widehat{\bF_q(C)}_{\nu_i},\nu_i\big)\simeq \big(\bF_q((t)),\nu\big)$, where $ \big(\bF_q((t)),\nu\big)$ is the formal Laurent field of $\bF_p$ with the valuation $\nu$ defined by  $\nu(\sum_{i=m}^{+\infty}a_it^i)=\min \{i\mid a_i\neq 0\}$.    Let $\varrho_i:\pi_1(X)\to \GL_N(\bF_q((t)))$ be the extension of ${\varrho_{C}}$ with respect to $\big(\widehat{\bF_q(C)}_{\nu_i},\nu_i\big)$.  
\begin{claim}\label{lem:simple}
There exists some $i$ such that $\varrho_i$ is unbounded.
\end{claim}
\begin{proof}
	This proof is in the same vein as \cref{lem:simple2} and we repeat  it here for   completeness. Assume by contradiction that $\varrho_i$ is bounded for each $i$. Then  after we replace $\varrho_i$ by some conjugation, we have $\varrho_i(\pi_1(X))\subset \GL_{N}(\bF_q[[t]])$.  For any matrix $A\in \GL_N(K)$,  we denote by $\chi(A)=T^N+\sigma_1(A)T^{N-1}+\cdots+\sigma_N(A)$ its characteristic polynomial.    Then $\sigma_j(\varrho_C(\gamma))\in \bF_q[C]$ for each $\gamma\in \pi_1(X)$.  
	Since we have assumed that $\varrho_i(\pi_1(X))\subset \GL_{N}(\bF_q[[t]])$ for each $i$, it follows that $\sigma_{j}(\varrho_i(\gamma))\in \bF_q[[t]]$ for each $i$.  Therefore, by the definition of $\varrho_i$, $\nu_i\big(\sigma_j(\varrho_C(\gamma))\big)\geq 0$ for each $i$.  It follows that $\sigma_j(\varrho_C(\gamma))$  extends to a regular function on $\overline{C}$, which is thus constant.  Hence $\pi_p(C_o)$ is a point, leading to a contradiction.     
\end{proof}  
 Let $\tau_i:\pi_1(X) \to \GL_N(\overline{\bF_q((t))})$ be the semisimplification of $\varrho_i$. 	
Note that  for the representation $f^*\tau_i:\pi_1(Y)\to \GL_N(\overline{\bF_q((t))})$,  its semisimplification $(f^*\tau_i)^{ss}$ is conjugate to $\sigma:\pi_1(Y)\to\GL_{N}(\overline{\bF_p})$. Note that  $\sigma(\pi_1(X))$ has  finite image.    Hence $(f^*\tau_i)^{ss}$ is bounded and by \Cref{lem:DY23bounded}, $f^*\tau_i$ is bounded.  By virtue of \cref{thm:KE}, $\iota^*\tau_i$ is also bounded, where $\iota:Z\to X$ is the inclusion. Since ${\rm Im}[\pi_1(Z)\to \pi_1(X)]$ is a finite index subgroup of $\pi_1(X)$, it follows that $\tau_i$ is bounded, and thus $\varrho_i$ is bounded by    \Cref{lem:DY23bounded} for any $i$.  This contradicts to \cref{lem:simple}.   The theorem is proved.
 \end{proof} 

\medspace

\noindent \textbf{Acknowledgements}. 
YD would like to thank Michel Brion for very helpful discussions.  This work was started when he visited University of Science and Technology of China (USTC) in July 2023, and he  expresses gratitude for the warm hospitality. YD is partially supported by the French Agence Nationale de la Recherche (ANR) under reference ANR-21-CE40-0010. YL expresses gratitude to Prof. Bing Wang for the support provided during his postdoctoral period at USTC.

\end{document}